\documentclass[11pt]{article}
\usepackage{amssymb,amsmath,amsfonts}

\newtheorem{theorem}{Theorem}[section]

\newtheorem{corollary}[theorem]{Corollary}

\newtheorem{definition}[theorem]{Definition}

\numberwithin{equation}{section}
\newenvironment{proof}{\par\noindent{\bf Proof.}}{$\square$\par\bigskip}
\input{tcilatex}
\begin{document}

\title{\textbf{The Hadamard Products\ for Bi-periodic Fibonacci and
Bi-periodic Lucas Generating Matrices}}
\author{\texttt{Arzu Coskun} and \texttt{Necati Taskara}}
\date{Department of Mathematics, Faculty of Science,\\
Selcuk University, Campus, 42075, Konya - Turkey \\
[0.3cm] \textit{arzucoskun58@gmail.com} and \textit{ntaskara@selcuk.edu.tr}}
\maketitle

\begin{abstract}
{\footnotesize {In this paper, firstly, we define the }}$Q_{q}$%
{\footnotesize {\textit{-generating matrix for} bi-periodic Fibonacci
polynomial. And we give nth power, determinant and some properties of the
bi-periodic Fibonacci polynomial by considering this matrix representation.
Also, we introduce the Hadamard products for bi-periodic Fibonacci }}$%
Q_{q}^{n}${\footnotesize {\ generating \textit{matrix and bi-periodic Lucas }%
}}$Q_{l}^{n}${\footnotesize {\ generating matrix of which entries is
bi-periodic Fibonacci and Lucas numbers. Then, we investigate some
properties of these products.}}


\end{abstract}

\section{Introduction and Preliminaries}

\qquad The special sequences and their properties have been investigated in
many articles and books (see, for example \cite%
{Bilgici,Coskun,Edson,Falcon,Hoggatt,Koshy}, \cite{UsluUygun}-\cite%
{YilmazCoskunTaskara} and the references cited therein). The Fibonacci and
Lucas numbers have attracted the attention of mathematicians because of
their intrinsic theory and applications. Fibonacci $\left\{ F_{n}\right\}
_{n\in \mathbb{N}}$ and Lucas $\left\{ L_{n}\right\} _{n\in \mathbb{N}}$
sequences was defined%
\begin{equation*}
F_{n+1}=F_{n}+F_{n-1},\text{ \ }L_{n+1}=L_{n}+L_{n-1}
\end{equation*}%
with initial conditions $F_{1}=\ F_{2}=1,$ $L_{1}=2,$ $L_{2}=1.$

Many authors have generalized Fibonacci sequence in different ways. In the
one of those generalizations, in \cite{YilmazCoskunTaskara}, we define the
bi-periodic Fibonacci $\left\{ q_{n}(x)\right\} _{n\in \mathbb{N}}$
polynomial as in the form%
\begin{equation}
q_{n}(x)=\left\{ 
\begin{array}{c}
axq_{n-1}(x)+q_{n-2}(x),\ \ \text{if }n\text{ is odd} \\ 
bxq_{n-1}(x)+q_{n-2}(x),\ \ \text{if }n\text{ is even}%
\end{array}%
\right.  \label{1.1}
\end{equation}%
where $q_{0}(x)=0,\ q_{1}(x)=1$ and $a,b$ are nonzero real numbers$.$

Also, the Binet formula and Cassini identity of this polynomial was given.

In \cite{Edson}, for $a$ and $b$ are nonzero real numbers, the authors
defined the bi-periodic Fibonacci $\left\{ q_{n}\right\} _{n\in \mathbb{N}}$
sequence%
\begin{equation*}
q_{n}=\left\{ 
\begin{array}{c}
aq_{n-1}+q_{n-2},\ \ \text{if }n\text{ is even} \\ 
bq_{n-1}+q_{n-2},\ \ \text{if }n\text{ is odd}\ 
\end{array}%
\right. \ ,
\end{equation*}%
where $q_{0}=0,\ q_{1}=1.$ In \cite{Bilgici}, for $a$ and $b$ are nonzero
real numbers, it is defined the bi-periodic Lucas $\left\{ l_{n}\right\}
_{n\in \mathbb{N}}$ sequence as%
\begin{equation*}
l_{n}=\left\{ 
\begin{array}{c}
bl_{n-1}+l_{n-2},\ \ \text{if }n\text{ is even} \\ 
al_{n-1}+l_{n-2},\ \ \text{if }n\text{ is odd}\ 
\end{array}%
\right. \ ,
\end{equation*}%
where $l_{0}=2,\ l_{1}=a.$ The author also gave in the following relations%
\begin{eqnarray}
(ab+4)q_{n} &=&l_{n+1}+l_{n-1},  \label{1.4} \\
l_{n} &=&q_{n+1}+q_{n-1},  \notag
\end{eqnarray}

\begin{eqnarray}
a^{1-\varepsilon (n)}b^{\varepsilon (n)}q_{n+1}q_{n-1}-a^{\varepsilon
(n)}b^{1-\varepsilon (n)}q_{n}^{2} &=&a(-1)^{n},  \label{1.5} \\
\left( \frac{b}{a}\right) ^{1-\varepsilon (n)}l_{n+1}l_{n-1}-\left( \frac{b}{%
a}\right) ^{\varepsilon (n)}l_{n}^{2} &=&(ab+4)(-1)^{n+1}  \notag
\end{eqnarray}%
where $\varepsilon (n)=n-2\left\lfloor \frac{n}{2}\right\rfloor .$

On the other hand, it have been studied the matrix representation of special
sequences (\cite{Morales,CoskunTaskara,GulecTaskara,Ocal,Sylvester}). In 
\cite{Sylvester}, Sylvester gave Fibonacci $Q$-matrix as%
\begin{equation*}
Q=\left[ 
\begin{array}{cc}
1 & 1 \\ 
1 & 0%
\end{array}%
\right] .
\end{equation*}

Then, he says that some properties of Fibonacci numbers can be founded by
using this matrix. Considering this matrix, in \cite{Hoggatt}, the author
obtained some equalities for $Q^{n}.$

And, using this properties in \cite{Nalli}, the authors defined the Hadamard
product $Q^{n}\circ Q^{-n}$. Similarly, in \cite{Tasci}, the author gave
some properties for the Hadamard products of its Adjoint matrix with a
square matrix.

In \cite{CoskunTaskara}, we defined the $Q_{l}$ matrix as%
\begin{equation}
Q_{l}=\left[ 
\begin{array}{cc}
a^{2}+2\frac{a}{b} & \frac{a^{2}}{b} \\ 
a & 2\frac{a}{b}%
\end{array}%
\right] .  \label{1.6}
\end{equation}

And, we gave%
\begin{equation}
Q_{l}^{n}=\left\{ 
\begin{array}{c}
\left( \frac{a}{b}\right) ^{n}(ab+4)^{\frac{n}{2}}\left[ 
\begin{array}{cc}
q_{n+1} & q_{n} \\ 
\frac{b}{a}q_{n} & q_{n-1}%
\end{array}%
\right] ,n\text{ }even \\ 
\left( \frac{a}{b}\right) ^{n}(ab+4)^{\frac{n-1}{2}}\left[ 
\begin{array}{cc}
l_{n+1} & l_{n} \\ 
\frac{b}{a}l_{n} & l_{n-1}%
\end{array}%
\right] ,n\text{ }odd%
\end{array}%
\right. ,\text{ }\det \left( \frac{a^{2}}{b^{2}}\left( ab+4\right) \right)
^{n}.  \label{1.7}
\end{equation}

This study consists of three sections. In Section 2, we define the $Q_{q}$%
-generating matrix as the first time in the literature. This matrix is
generalization form of well-known $Q$-Fibonacci matrix. By using this
generalized matrix, we find some equalities for bi-periodic Fibonacci
polynomial. In the third part of our work, we define the Hadamard products $%
Q_{q}^{n}\circ Q_{q}^{-n}$\textbf{\ \ }and\textbf{\ }$Q_{l}^{n}\circ
Q_{l}^{-n}$ . And, we get some properties of these generalized Hadamard
products.

\section{The properties of $Q_{q}$-generating matrix}

\begin{definition}
\label{def1} Bi-periodic Fibonacci $Q_{q}$-generating matrix is defined by%
\begin{equation}
Q_{q}=\left[ 
\begin{array}{cc}
bx & \frac{b}{a} \\ 
1 & 0%
\end{array}%
\right] .  \label{2.1}
\end{equation}
\end{definition}

\begin{theorem}
\label{teo1} Let $Q_{q}$-generating matrix be as in equation (\ref{2.1}%
).Then, we have%
\begin{equation}
Q_{q}^{n}=\left( \frac{b}{a}\right) ^{\left\lfloor \frac{n}{2}\right\rfloor }%
\left[ 
\begin{array}{cc}
\left( \frac{b}{a}\right) ^{\varepsilon (n)}q_{n+1}(x) & \frac{b}{a}q_{n}(x)
\\ 
q_{n}(x) & \left( \frac{b}{a}\right) ^{\varepsilon (n)}q_{n-1}(x)%
\end{array}%
\right]  \label{2.2}
\end{equation}%
where $\varepsilon (n)=n-\left\lfloor \frac{n}{2}\right\rfloor $ and $%
q_{n}(x)$ is $n$th bi-periodic Fibonacci polynomial.
\end{theorem}

\begin{proof}
We use mathematical induction on $n$, we can write%
\begin{equation*}
Q_{q}=\left[ 
\begin{array}{cc}
bx & \frac{b}{a} \\ 
1 & 0%
\end{array}%
\right] =\left( \frac{b}{a}\right) ^{\left\lfloor \frac{1}{2}\right\rfloor }%
\left[ 
\begin{array}{cc}
\left( \frac{b}{a}\right) ^{\varepsilon (1)}q_{2}(x) & \frac{b}{a}q_{1}(x)
\\ 
q_{1}(x) & \left( \frac{b}{a}\right) ^{\varepsilon (1)}q_{0}(x)%
\end{array}%
\right] .
\end{equation*}%
\begin{equation*}
Q_{q}^{2}=\left( 
\begin{array}{cc}
b^{2}x^{2}+\frac{b}{a} & \frac{b^{2}}{a}x \\ 
bx & \frac{b}{a}%
\end{array}%
\right) =\left( \frac{b}{a}\right) ^{\left\lfloor \frac{2}{2}\right\rfloor }%
\left[ 
\begin{array}{cc}
\left( \frac{b}{a}\right) ^{\varepsilon (2)}q_{3}(x) & \frac{b}{a}q_{2}(x)
\\ 
q_{2}(x) & \left( \frac{b}{a}\right) ^{\varepsilon (2)}q_{1}(x)%
\end{array}%
\right] .
\end{equation*}%
which show that the equation (\ref{2.2}) is true for $n=1$ and $n=2$. Now,
we suppose that equation (\ref{2.2}) is true for $n=k$, that is%
\begin{equation*}
Q_{q}^{k}=\left( \frac{b}{a}\right) ^{\left\lfloor \frac{k}{2}\right\rfloor }%
\left[ 
\begin{array}{cc}
\left( \frac{b}{a}\right) ^{\varepsilon (k)}q_{k+1}(x) & \frac{b}{a}q_{k}(x)
\\ 
q_{k}(x) & \left( \frac{b}{a}\right) ^{\varepsilon (k)}q_{k-1}(x)%
\end{array}%
\right] .
\end{equation*}

If we supposed that k is even, by using properties of the bi-periodic
Fibonacci polynomial, we get%
\begin{equation*}
Q_{q}^{k+2}=Q_{q}^{2}Q_{q}^{k}=\left( \frac{b}{a}\right) ^{\frac{k+2}{2}}%
\left[ 
\begin{array}{cc}
q_{k+3}(x) & \frac{b}{a}q_{k+2}(x) \\ 
q_{k+2}(x) & q_{k+1}(x)%
\end{array}%
\right] .
\end{equation*}%
Similarly, for k is odd, we can write%
\begin{equation*}
Q_{q}^{k+2}=Q_{q}^{2}Q_{q}^{k}=\left( \frac{b}{a}\right) ^{\frac{k+1}{2}}%
\left[ 
\begin{array}{cc}
\frac{b}{a}q_{k+3}(x) & \frac{b}{a}q_{k+2}(x) \\ 
q_{k+2}(x) & \frac{b}{a}q_{k+1}(x)%
\end{array}%
\right] .
\end{equation*}%
By combining this equalities, we obtain%
\begin{equation*}
Q_{q}^{k+2}=\left( \frac{b}{a}\right) ^{\left\lfloor \frac{k+2}{2}%
\right\rfloor }\left[ 
\begin{array}{cc}
\left( \frac{b}{a}\right) ^{\varepsilon (k+2)}q_{k+3}(x) & \frac{b}{a}%
q_{k+2}(x) \\ 
q_{k+2}(x) & \left( \frac{b}{a}\right) ^{\varepsilon (k+2)}q_{k+1}(x)%
\end{array}%
\right] .
\end{equation*}
\end{proof}

\begin{corollary}
Let $Q_{q}^{n}$ be as in equation (\ref{2.1}). Then the following equality
is true for all positive integers%
\begin{equation}
\det \left( Q_{q}^{n}\right) =\left( -\frac{b}{a}\right) ^{n}.  \label{2.3}
\end{equation}
\end{corollary}

\begin{proof}
By using the Cassini identity for bi-periodic Fibonacci polnomial, the
desired result is obtained.
\end{proof}

The Binet formula for bi-periodic Fibonacci polynomial, given in \cite%
{YilmazCoskunTaskara}, can also be obtained by using $Q_{q}$ matrix.

\begin{theorem}
\label{teo3} Let $n$ be any integer. The Binet formula of bi-periodic
Fibonacci polynomial is%
\begin{equation*}
q_{n}(x)=\frac{a^{1-\varepsilon (n)}}{\left( ab\right) ^{\left\lfloor \frac{n%
}{2}\right\rfloor }x^{n-1}}\frac{\alpha ^{n}(x)-\beta ^{n}(x)}{\alpha
(x)-\beta (x)}
\end{equation*}%
where $\alpha (x)$ and $\beta (x)$ are roots of $r^{2}-abx^{2}r-abx^{2}=0.$
\end{theorem}

\begin{proof}
Let the $Q_{q}$ matrix be as in equation (\ref{2.1}). Characteristic
equation of $Q_{q}$-generating matrix is $\lambda ^{2}-bx\lambda -\frac{b}{a}%
=0$. Then, eigenvalues and eigenvectors of the $Q_{q}$ are%
\begin{equation*}
\lambda _{1}=\frac{\alpha (x)}{ax},\text{ }\lambda _{2}=\frac{\beta (x)}{ax},%
\text{ }u_{1}=\left[ \frac{b}{a},-\frac{\beta (x)}{ax}\right] ,\text{\ }%
u_{2}=\left[ \frac{b}{a},-\frac{\alpha (x)}{ax}\right] .
\end{equation*}%
The generating matrix can be diagonalized by using%
\begin{equation*}
V=U^{-1}Q_{q}U
\end{equation*}%
which%
\begin{equation*}
U=\left[ 
\begin{array}{cc}
u_{1}^{T} & u_{2}^{T}%
\end{array}%
\right] =\left[ 
\begin{array}{cc}
\frac{b}{a} & \frac{b}{a} \\ 
-\frac{\beta (x)}{ax} & -\frac{\alpha (x)}{ax}%
\end{array}%
\right] ,\text{ }V=diag\left[ \lambda _{1},\lambda _{2}\right] =\left[ 
\begin{array}{cc}
\frac{\alpha }{ax} & 0 \\ 
0 & \frac{\beta }{ax}%
\end{array}%
\right] .
\end{equation*}%
From properties of similar matrices, for $n$ is any integer, we obtain%
\begin{equation*}
Q_{q}^{n}=UV^{n}U^{-1}.
\end{equation*}%
Thus we get%
\begin{eqnarray*}
Q_{q}^{n} &=&\left[ 
\begin{array}{cc}
\frac{b}{a} & \frac{b}{a} \\ 
-\frac{\beta (x)}{ax} & -\frac{\alpha (x)}{ax}%
\end{array}%
\right] \left[ 
\begin{array}{cc}
\left( \frac{\alpha (x)}{ax}\right) ^{n} & 0 \\ 
0 & \left( \frac{\beta (x)}{ax}\right) ^{n}%
\end{array}%
\right] \frac{a^{2}x}{b\left( \beta (x)-\alpha (x)\right) }\left[ 
\begin{array}{cc}
-\frac{\alpha (x)}{ax} & -\frac{b}{a} \\ 
\frac{\beta (x)}{ax} & \frac{b}{a}%
\end{array}%
\right] \\
&=&\frac{1}{\alpha (x)-\beta (x)}\left[ 
\begin{array}{cc}
\frac{\alpha ^{n+1}(x)-\beta ^{n+1}(x)}{a^{n}x^{n}} & \frac{b\left( \alpha
^{n}(x)-\beta ^{n}(x)\right) }{a^{n}x^{n-1}} \\ 
\frac{\alpha ^{n}(x)-\beta ^{n}(x)}{a^{n-1}x^{n-1}} & \frac{b\left( \alpha
^{n-1}(x)-\beta ^{n-1}(x)\right) }{a^{n-1}x^{n-2}}%
\end{array}%
\right] .
\end{eqnarray*}%
Taking into account the Theorem (\ref{teo1}), we have%
\begin{equation*}
\frac{1}{\alpha (x)-\beta (x)}\frac{\alpha ^{n}(x)-\beta ^{n}(x)}{%
a^{n-1}x^{n-1}}=\left( \frac{b}{a}\right) ^{\left\lfloor \frac{n}{2}%
\right\rfloor }q_{n}(x).
\end{equation*}%
Making the necessary arrangements, the desired result is obtained.
\end{proof}

In the following, we give some properties of bi-periodic Fibonacci
polynomial.

\begin{theorem}
\label{teo4} For every $m,n\in 
%TCIMACRO{\U{2115} }%
%BeginExpansion
\mathbb{N}
%EndExpansion
,$ the following statements are valid
\end{theorem}

\begin{enumerate}
\item[$i)$] $q_{m+n}(x)=\left\{ 
\begin{array}{c}
\text{ \ \ \ \ \ \ \ \ \ \ \ \ \ }q_{m+1}(x)q_{n}(x)+q_{m}(x)q_{n-1}(x),%
\text{ \ \ \ \ \ \ \ \ \ }m+n\text{ even} \\ 
\left( \frac{b}{a}\right) ^{\varepsilon (m)}q_{m+1}(x)q_{n}(x)+\left( \frac{b%
}{a}\right) ^{\varepsilon (n)}q_{m}(x)q_{n-1}(x),\text{ \ \ }m+n\text{ odd}%
\end{array}%
\right. ,$

\item[$ii)$] $q_{m+n}(x)=\left\{ 
\begin{array}{c}
\text{ \ \ \ \ \ \ \ \ \ \ \ \ }q_{m}(x)q_{n+1}(x)+q_{m-1}(x)q_{n}(x),\text{
\ \ \ \ \ \ \ \ \ \ }m+n\text{ even} \\ 
\left( \frac{b}{a}\right) ^{\varepsilon (n)}q_{m}(x)q_{n+1}(x)+\left( \frac{b%
}{a}\right) ^{\varepsilon (m)}q_{m-1}(x)q_{n}(x),\text{ \ \ }m+n\text{ odd}%
\end{array}%
\right. ,$

\item[$iii)$] $q_{m-n}(x)=\left\{ 
\begin{array}{c}
\text{ \ \ \ \ \ }\left( -1\right) ^{n+1}\left\{
q_{m+1}(x)q_{n}(x)-q_{m}(x)q_{n+1}(x)\right\} ,\text{ \ \ \ \ \ \ \ }m+n%
\text{ even} \\ 
\left( -\frac{b}{a}\right) ^{\varepsilon (m)}q_{m+1}(x)q_{n}(x)+\left( -%
\frac{b}{a}\right) ^{\varepsilon (n)}q_{m}(x)q_{n+1}(x),\text{ \ \ }m+n\text{
odd}%
\end{array}%
\right. ,$

\item[$iv)$] $q_{m-n}(x)=\left\{ 
\begin{array}{c}
\text{ \ \ \ \ \ }\left( -1\right) ^{n+1}\left\{
q_{m-1}(x)q_{n}(x)-q_{m}(x)q_{n-1}(x)\right\} ,\text{ \ \ \ \ \ \ \ }m+n%
\text{ even} \\ 
\left( -\frac{b}{a}\right) ^{\varepsilon (n)}q_{m}(x)q_{n-1}(x)+\left( -%
\frac{b}{a}\right) ^{\varepsilon (m)}q_{m-1}(x)q_{n}(x),\text{ \ \ }m+n\text{
odd}%
\end{array}%
\right. .$
\end{enumerate}

\begin{proof}
By using equation (\ref{2.2}), $Q_{q}^{m+n}$\ can be written as%
\begin{equation}
Q_{q}^{m+n}=\left( \frac{b}{a}\right) ^{\left\lfloor \frac{m+n}{2}%
\right\rfloor }\left[ 
\begin{array}{cc}
\left( \frac{b}{a}\right) ^{\varepsilon (m+n)}q_{m+n+1}(x) & \frac{b}{a}%
q_{m+n}(x) \\ 
q_{m+n}(x) & \left( \frac{b}{a}\right) ^{\varepsilon (n)}q_{m+n-1}(x)%
\end{array}%
\right] .  \label{2.5}
\end{equation}%
For case of odd $m$ and $n$, we can write%
\begin{equation}
Q_{q}^{m}Q_{q}^{n}=\left( \frac{b}{a}\right) ^{\frac{m+n}{2}-1}\left[ 
\begin{array}{cc}
\left( \frac{b}{a}\right) ^{2}q_{m+1}(x)q_{n+1}(x)+\frac{b}{a}%
q_{m}(x)q_{n}(x) & \left( \frac{b}{a}\right) ^{2}\left\{
q_{m+1}(x)q_{n}(x)+q_{m}(x)q_{n-1}(x)\right\} \\ 
\frac{b}{a}\left\{ q_{m}(x)q_{n+1}(x)+q_{m-1}(x)q_{n}(x)\right\} & \frac{b}{a%
}q_{m}(x)q_{n}(x)+\left( \frac{b}{a}\right) ^{2}q_{m-1}(x)q_{n-1}(x)%
\end{array}%
\right] .  \label{2.6}
\end{equation}%
If we compare the $1st$ row and $2nd$ column entries of the matrices (\ref%
{2.5}) and (\ref{2.6}), we get%
\begin{equation*}
q_{m+n}(x)=q_{m+1}(x)q_{n}(x)+q_{m}(x)q_{n-1}(x).
\end{equation*}%
On the other hand, comparing the $2nd$ row and $1st$ column, we obtain%
\begin{equation*}
q_{m+n}(x)=q_{m}(x)q_{n+1}(x)+q_{m-1}(x)q_{n}(x).
\end{equation*}%
Similarly, for the case of even $m$ and $n$, we have%
\begin{eqnarray*}
q_{m+n}(x) &=&q_{m+1}(x)q_{n}(x)+q_{m}(x)q_{n-1}(x), \\
q_{m+n}(x) &=&q_{m}(x)q_{n+1}(x)+q_{m-1}(x)q_{n}(x).
\end{eqnarray*}%
And, for the case of odd $m$ and even $n$ (or case of even $m$ and odd $n$),
we have%
\begin{eqnarray*}
q_{m+n}(x) &=&\frac{b}{a}q_{m+1}(x)q_{n}(x)+q_{m}(x)q_{n-1}(x), \\
q_{m+n}(x) &=&q_{m}(x)q_{n+1}(x)+\frac{b}{a}q_{m-1}(x)q_{n}(x).
\end{eqnarray*}

Thus, the proof of $i)$ and $ii)$ is obtained.

Now, we give the proof of $iii)$ and $iv).$ By calculating inverse of the
matrix $Q_{q}^{n}$ in (\ref{2.2}), we conclude%
\begin{equation}
Q_{q}^{-n}=\left\{ 
\begin{array}{c}
\text{ \ \ \ \ \ }\left( \frac{b}{a}\right) ^{-\frac{n}{2}}\left[ 
\begin{array}{cc}
q_{n-1}(x) & -\frac{b}{a}q_{n}(x) \\ 
-q_{n}(x) & q_{n+1}(x)%
\end{array}%
\right] ,\text{ \ \ \ \ \ }n\text{ even} \\ 
\left( \frac{b}{a}\right) ^{-\frac{n+1}{2}}\left[ 
\begin{array}{cc}
-\frac{b}{a}q_{n-1}(x) & \frac{b}{a}q_{n}(x) \\ 
q_{n}(x) & -\frac{b}{a}q_{n+1}(x)%
\end{array}%
\right] ,\text{ \ \ }n\text{ odd}%
\end{array}%
\right. .  \label{2.7}
\end{equation}

Benefitting from the equality $Q_{q}^{m-n}=Q_{q}^{m}Q_{q}^{-n}$ and by
comparing the entries, the desired result can be obtained. That is, for the
case of odd $m$ and $n$, we get%
\begin{eqnarray*}
q_{m-n}(x) &=&q_{m+1}(x)q_{n}(x)-q_{m}(x)q_{n-1}(x), \\
q_{m-n}(x) &=&q_{m-1}(x)q_{n}(x)-q_{m}(x)q_{n-1}(x).
\end{eqnarray*}%
Similarly, for the case of even $m$ and $n$, we obtain%
\begin{eqnarray*}
q_{m-n}(x) &=&q_{m}(x)q_{n+1}(x)-q_{m+1}(x)q_{n}(x), \\
q_{m-n}(x) &=&q_{m}(x)q_{n-1}(x)-q_{m-1}(x)q_{n}(x).
\end{eqnarray*}%
And, for the case of odd $m$ and even $n$ (or case of even $m$ and odd $n$),
we have%
\begin{eqnarray*}
q_{m-n}(x) &=&-\frac{b}{a}q_{m+1}(x)q_{n}(x)+q_{m}(x)q_{n+1}(x), \\
q_{m-n}(x) &=&q_{m}(x)q_{n-1}(x)-\frac{b}{a}q_{m-1}(x)q_{n}(x).
\end{eqnarray*}%
Thus, we have the desired expressions.
\end{proof}

\section{On the \textbf{Hadamard Products }$Q_{q}^{n}\circ Q_{q}^{-n}$%
\textbf{\ \ and }$Q_{l}^{n}\circ Q_{l}^{-n}$\textbf{\ for Bi-periodic
Fibonacci and Bi-periodic Lucas Generating Matrices}}

We will accept $x=1$ in $Q_{q}^{n}$ matrix throughout this section. Then, we
can give the following Theorem for the generating matrices of bi-periodic
Fibonacci and Lucas numbers.

\begin{theorem}
\label{nth power} For any integer $n,$ we have
\end{theorem}

\begin{itemize}
\item[$i)$] $Q_{q}^{n}=bQ_{q}^{n-1}+\frac{b}{a}Q_{q}^{n-2},$

\item[$ii)$] $Q_{l}^{n}=\frac{b}{a(ab+4)}Q_{l}^{n+1}+\frac{a}{b}Q_{l}^{n-1}.$
\end{itemize}

\begin{proof}
We will omit the proof of $(ii)$ because it is similar to the proof of $(i).$

By considering the Equation (\ref{2.2}), for even $n$, we can write%
\begin{eqnarray*}
Q_{q}^{n} &=&\left( \frac{b}{a}\right) ^{\frac{n}{2}}\left[ 
\begin{array}{cc}
q_{n+1} & \frac{b}{a}q_{n} \\ 
q_{n} & q_{n-1}%
\end{array}%
\right] =\left( \frac{b}{a}\right) ^{\frac{n}{2}}\left[ 
\begin{array}{cc}
bq_{n}+q_{n-1} & \frac{b}{a}\left( aq_{n-1}+q_{n-2}\right) \\ 
aq_{n-1}+q_{n-2} & bq_{n-2}+q_{n-3}%
\end{array}%
\right] \\
&=&b\left( \frac{b}{a}\right) ^{\frac{n-2}{2}}\left[ 
\begin{array}{cc}
\frac{b}{a}q_{n} & \frac{b}{a}aq_{n-1} \\ 
q_{n-1} & \frac{b}{a}q_{n-2}%
\end{array}%
\right] +\frac{b}{a}\left( \frac{b}{a}\right) ^{\frac{n-2}{2}}\left[ 
\begin{array}{cc}
q_{n-1} & \frac{b}{a}q_{n-2} \\ 
q_{n-2} & q_{n-3}%
\end{array}%
\right] .
\end{eqnarray*}

Similarly, for odd $n$, we have%
\begin{eqnarray*}
Q_{q}^{n} &=&\left( \frac{b}{a}\right) ^{\frac{n-1}{2}}\left[ 
\begin{array}{cc}
\frac{b}{a}q_{n+1} & \frac{b}{a}q_{n} \\ 
q_{n} & \frac{b}{a}q_{n-1}%
\end{array}%
\right] =\left( \frac{b}{a}\right) ^{\frac{n-1}{2}}\left[ 
\begin{array}{cc}
\frac{b}{a}\left( aq_{n}+q_{n-1}\right) & \frac{b}{a}\left(
bq_{n-1}+q_{n-2}\right) \\ 
bq_{n-1}+q_{n-2} & \frac{b}{a}\left( aq_{n-2}+q_{n-3}\right)%
\end{array}%
\right] \\
&=&b\left( \frac{b}{a}\right) ^{\frac{n-1}{2}}\left[ 
\begin{array}{cc}
q_{n} & \frac{b}{a}q_{n-1} \\ 
q_{n-1} & q_{n-2}%
\end{array}%
\right] +\frac{b}{a}\left( \frac{b}{a}\right) ^{\frac{n-2}{2}}\left[ 
\begin{array}{cc}
\frac{b}{a}q_{n-1} & \frac{b}{a}q_{n-2} \\ 
q_{n-2} & \frac{b}{a}q_{n-3}%
\end{array}%
\right] .
\end{eqnarray*}

Thus, the desired equality is obtained.
\end{proof}

\qquad Now, we define the Hadamard products for bi-periodic Fibonacci and
Lucas generating matrices by considering the determinants. Thus, we can write%
\begin{equation}
Q_{q}^{n}\circ Q_{q}^{-n}=\left\{ 
\begin{array}{c}
\text{ \ \ \ }\left( \frac{a}{b}\right) ^{n}\text{\ }\left( Q_{q}^{n}\circ
adjQ_{q}^{n}\right) ,n\text{ }even \\ 
-\left( \frac{a}{b}\right) ^{n}(Q_{q}^{n}\circ adjQ_{q}^{n}),n\text{ }odd%
\end{array}%
\right.  \label{3.1}
\end{equation}

and%
\begin{equation}
Q_{l}^{n}\circ Q_{l}^{-n}=\left\{ 
\begin{array}{c}
\text{ \ \ \ \ }\left( \frac{b^{2}}{a^{2}\left( ab+4\right) }\right)
^{n}\left( Q_{l}^{n}\circ adjQ_{l}^{n}\right) ,n\text{ }even \\ 
-\left( \frac{b^{2}}{a^{2}\left( ab+4\right) }\right) ^{n}(Q_{l}^{n}\circ
adjQ_{l}^{n}),n\text{ }odd%
\end{array}%
\right. .  \label{3.2}
\end{equation}

In the following theorem, we give the determinants of the Hadamard products.

\begin{theorem}
\label{det} For any integer $n,$ we have%
\begin{eqnarray*}
\det (Q_{q}^{n}\circ Q_{q}^{-n}) &=&\left\{ 
\begin{array}{c}
1+2\frac{b}{a}q_{n}^{2},\text{ }n\text{ }even \\ 
1-2q_{n}^{2},n\text{ }odd%
\end{array}%
\right. , \\
\det (Q_{l}^{n}\circ Q_{l}^{-n}) &=&\left\{ 
\begin{array}{c}
1+2\frac{b}{a}q_{n}^{2},\text{ }n\text{ }even \\ 
1+2\frac{b}{a\left( ab+4\right) }l_{n}^{2},n\text{ }odd%
\end{array}%
\right. .
\end{eqnarray*}
\end{theorem}

\begin{proof}
We are doing the proof for $\det (Q_{q}^{n}\circ Q_{q}^{-n})$ because the
proof of $\det (Q_{l}^{n}\circ Q_{l}^{-n})$ can be done in similar way. By
using the Equation (\ref{3.1}), for even $n$, we can write%
\begin{equation*}
Q_{q}^{n}\circ Q_{q}^{-n}=\left[ 
\begin{array}{cc}
q_{n+1}q_{n-1} & -\frac{b^{2}}{a^{2}}q_{n}^{2} \\ 
-q_{n}^{2} & q_{n+1}q_{n-1}%
\end{array}%
\right] .
\end{equation*}

And, we have%
\begin{eqnarray*}
\det (Q_{q}^{n}\circ Q_{q}^{-n}) &=&\det \left[ 
\begin{array}{cc}
q_{n+1}q_{n-1} & -\frac{b^{2}}{a^{2}}q_{n}^{2} \\ 
-q_{n}^{2} & q_{n+1}q_{n-1}%
\end{array}%
\right] \\
&=&\left( q_{n+1}q_{n-1}-\frac{b}{a}q_{n}^{2}\right) \left( q_{n+1}q_{n-1}+%
\frac{b}{a}q_{n}^{2}\right) \\
&=&\left( q_{n+1}q_{n-1}+\frac{b}{a}q_{n}^{2}\right) =1+2\frac{b}{a}%
q_{n}^{2}.
\end{eqnarray*}

Similarly, for odd $n$, we get%
\begin{equation*}
Q_{q}^{n}\circ Q_{q}^{-n}=\left[ 
\begin{array}{cc}
-\frac{b}{a}q_{n+1}q_{n-1} & \frac{b}{a}q_{n}^{2} \\ 
\frac{a}{b}q_{n}^{2} & -\frac{b}{a}q_{n+1}q_{n-1}%
\end{array}%
\right] .
\end{equation*}%
\begin{eqnarray*}
\det (Q_{q}^{n}\circ Q_{q}^{-n}) &=&\det \left[ 
\begin{array}{cc}
-\frac{b}{a}q_{n+1}q_{n-1} & \frac{b}{a}q_{n}^{2} \\ 
\frac{a}{b}q_{n}^{2} & -\frac{b}{a}q_{n+1}q_{n-1}%
\end{array}%
\right] \\
&=&\left( \frac{b}{a}q_{n+1}q_{n-1}-q_{n}^{2}\right) \left( \frac{b}{a}%
q_{n+1}q_{n-1}+q_{n}^{2}\right) \\
&=&\left( -1\right) ^{n}\left( q_{n+1}q_{n-1}+\frac{b}{a}q_{n}^{2}\right)
=1+2\frac{b}{a}q_{n}^{2}.
\end{eqnarray*}

Thus, we get the desired result.
\end{proof}

Using the above definition and theorem, we obtain traces, eigenvalues,
eigenvectors and inverses for Hadamard products as in the following.

\begin{corollary}
\label{trace} We have
\end{corollary}

\begin{description}
\item[$i)$] \textit{trace}$(Q_{q}^{n}\circ Q_{q}^{-n})=\left\{ 
\begin{array}{c}
2\left( 1+\frac{b}{a}q_{n}^{2}\right) ,\text{ }n\text{ }even \\ 
2\left( 1-q_{n}^{2}\right) ,n\text{ }odd%
\end{array}%
\right. ,$
\end{description}

\qquad \textit{trace}$(Q_{l}^{n}\circ Q_{l}^{-n})=\left\{ 
\begin{array}{c}
2\left( 1+\frac{b}{a}q_{n}^{2}\right) ,\text{ }n\text{ }even \\ 
2\left( 1+\frac{b}{a\left( ab+4\right) }l_{n}^{2}\right) ,n\text{ }odd%
\end{array}%
\right. $.

\begin{description}
\item[$ii)$] The eigenvalues of the matrix $Q_{q}^{n}\circ Q_{q}^{-n}$ are
given as%
\begin{eqnarray*}
\lambda _{1} &=&1,\text{ }\lambda _{2}=1+2\frac{b}{a}q_{n}^{2},\text{ for
even }n \\
\lambda _{1} &=&1,\text{ }\lambda _{2}=1-2q_{n}^{2},\text{ for odd }n
\end{eqnarray*}%
and the eigenvalues of the matrix $Q_{l}^{n}\circ Q_{l}^{-n}$ are given as%
\begin{eqnarray*}
\lambda _{3} &=&1,\text{ }\lambda _{4}=1+2\frac{b}{a}q_{n}^{2},\text{ 
\textit{for even} }n \\
\lambda _{3} &=&1,\text{ }\lambda _{4}=1+2\frac{b}{a\left( ab+4\right) }%
l_{n}^{2},\text{ \textit{for odd} }n
\end{eqnarray*}

\item[$iii)$] The eigenvectors corresponding to the eigenvalues of the
matrices $Q_{q}^{n}\circ Q_{q}^{-n}$ and $Q_{l}^{n}\circ Q_{l}^{-n}$ are%
\begin{equation*}
y_{1}=\left[ 
\begin{array}{c}
\frac{b}{a} \\ 
1%
\end{array}%
\right] ,\text{ }y_{2}=\left[ 
\begin{array}{c}
-\frac{b}{a} \\ 
1%
\end{array}%
\right] \text{ and }y_{3}=\left[ 
\begin{array}{c}
1 \\ 
\frac{b}{a}%
\end{array}%
\right] ,\text{ }y_{4}=\left[ 
\begin{array}{c}
1 \\ 
-\frac{b}{a}%
\end{array}%
\right]
\end{equation*}%
respectively.

\item[$iv)$] The matrices $Q_{q}^{n}\circ Q_{q}^{-n}$ and $Q_{l}^{n}\circ
Q_{l}^{-n}$ are invertible and 
\begin{eqnarray*}
\left( Q_{q}^{n}\circ Q_{q}^{-n}\right) ^{-1} &=&\left( Q_{l}^{n}\circ
Q_{l}^{-n}\right) ^{-1}=\left[ 
\begin{array}{cc}
1-\frac{\frac{b}{a}q_{n}^{2}}{1+2\frac{b}{a}q_{n}^{2}} & \frac{\frac{b^{2}}{%
a^{2}}q_{n}^{2}}{1+2\frac{b}{a}q_{n}^{2}} \\ 
\frac{q_{n}^{2}}{1+2\frac{b}{a}q_{n}^{2}} & 1-\frac{\frac{b}{a}q_{n}^{2}}{1+2%
\frac{b}{a}q_{n}^{2}}%
\end{array}%
\right] ,\text{ for even }n \\
\left( Q_{q}^{n}\circ Q_{q}^{-n}\right) ^{-1} &=&\left[ 
\begin{array}{cc}
1+\frac{q_{n}^{2}}{1-2q_{n}^{2}} & \frac{-\frac{b}{a}q_{n}^{2}}{1-2q_{n}^{2}}
\\ 
\frac{-\frac{a}{b}q_{n}^{2}}{1-2q_{n}^{2}} & 1-\frac{q_{n}^{2}}{1-2q_{n}^{2}}%
\end{array}%
\right] ,\text{ for odd }n \\
\left( Q_{l}^{n}\circ Q_{l}^{-n}\right) ^{-1} &=&\left[ 
\begin{array}{cc}
1-\frac{\frac{b}{a(ab+4)}l_{n}^{2}}{1+2\frac{b}{a\left( ab+4\right) }%
l_{n}^{2}} & \frac{\frac{1}{ab+4}l_{n}^{2}}{1+2\frac{b}{a\left( ab+4\right) }%
l_{n}^{2}} \\ 
\frac{\frac{b^{2}}{a^{2}\left( ab+4\right) }l_{n}^{2}}{1+2\frac{b}{a\left(
ab+4\right) }l_{n}^{2}} & 1-\frac{\frac{b}{a(ab+4)}l_{n}^{2}}{1+2\frac{b}{%
a\left( ab+4\right) }l_{n}^{2}}%
\end{array}%
\right] ,\text{ for odd }n
\end{eqnarray*}
\end{description}

\section{Conclusion}

In Section 2, we present some properties of bi-periodic Fibonacci polynomial
by using the $Q_{q}$ matrix. Then, we express that well-known matrix
representations are special cases of this generalized matrices. \qquad

If we choose $x=1,$ then we get generating matrix for bi-periodic Fibonacci
sequence and properties of this sequence.

Thus, if we choose the different values of $a$ and $b$, then we obtain
generating matrices for well-known matrix sequences in the literature:

\begin{itemize}
\item If we replace $a=b=1$ in $Q_{q}$, we obtain the generating matrices
for Fibonacci sequence.

\item If we replace $a=b=2$ in $Q_{q}$, we obtain the generating matrices
for Pell sequence.

\item If we replace $a=b=k$ in $Q_{q}$, we obtain the generating matrices
for $k$-Fibonacci sequence.
\end{itemize}

And, in Section 3, we define the Hadamard products for the generating
matrices of bi-periodic Fibonacci and Lucas sequence and give some
properties of these new matrices. By taking into account these generalized
matrices, it also can be obtained properties of some special matrices.
Namely, for different values of $a$ and $b$, we can rewrite some properties
for the well-known matrices in the literature such as Fibonacci $Q$.

\end{document}